\documentclass[a4paper,11pt,reqno,noindent]{amsart}
\usepackage[centertags]{amsmath}
\usepackage{amsfonts,amssymb,amsthm} 
\usepackage{hyperref}

 \hypersetup{
     pdfmenubar=false,        
     pdfnewwindow=true,      
     colorlinks=false,       
     linkcolor=blue,          
     citecolor=blue,        
     filecolor=magenta,      
     urlcolor=cyan           
 }
\usepackage{graphicx} 

\usepackage[english]{babel}
\usepackage{newlfont}
\usepackage{color}
\usepackage[body={15cm,21.5cm},centering]{geometry} 
\usepackage{fancyhdr}
\usepackage{esint}
\usepackage{enumerate}
\usepackage[numbers,sort]{natbib}
\usepackage{setspace}

\usepackage[active]{srcltx}

\newtheorem{theorem}{Theorem}
\newtheorem{lemma}{Lemma}
\newtheorem{proposition}{Proposition}

\newtheorem{remark}{Remark}
\newtheorem{definition}{Definition}

\theoremstyle{definition}

\newcommand{\dist}{\operatorname{dist}}
\newcommand{\supp}{\operatorname{supp}}

\newcommand{\e}{\varepsilon}

\newcommand{\R}{\mathbb{R}}
\newcommand{\Z}{\mathbb{Z}}
\newcommand{\N}{\mathbb{N}}

\newcommand{\diam}{\mathrm{diam}\,}
\renewcommand{\d}{\mathrm{d}}
\renewcommand{\div}{\mathrm{div}\,}
\renewcommand{\L}{\mathbb{L}}

\newcommand{\id}{\mathrm{Id}}

\newcommand{\bd}{\mathrm{dim}_{\mathrm{box}}}

\newcommand\sgn{\mathrm{sgn}}

\renewcommand{\H}{\mathcal{H}}

\renewcommand{\L}{{\mathcal L}}
\newcommand\wto{\rightharpoonup}

\newcommand{\Pf}{\mathrm{Pf}(\Omega)}
\newcommand{\Sym}{\mathrm{Sym}}

\newcommand{\cof}{\mathrm{cof}\,}

\title{Extrinsic curvature of codimension one isometric immersions with H\"older continuous derivatives}
\date{\today}
\author[S. Behr] {S{\"o}ren Behr}
\address[S{\"o}ren Behr]{Hausdorff Center for Mathematics \& Institute for Applied Mathematics, Bonn, Germany}
\email{s6sobehr@uni-bonn.de}
\author[H. Olbermann] {Heiner Olbermann}
\address[Heiner Olbermann]{Hausdorff Center for Mathematics \& Institute for Applied Mathematics, Bonn, Germany}
\email{heiner.olbermann@hcm.uni-bonn.de}

\begin{document}
\maketitle
\begin{abstract}
We prove that if $n$ is even, $(M,g)$ is a compact $n$-dimensional Riemannian manifold
whose Pfaffian form is a positive multiple of the volume form, and $y\in
C^{1,\alpha}(M;\R^{n+1})$ is an  isometric immersion with $n/(n+1)< \alpha\leq
1$, then $y(M)$ is a surface of bounded extrinsic curvature. This is proved by
showing that extrinsic curvature, defined by a suitable pull-back of the volume
form on the $n$-sphere via the Gauss map, is identical to intrinsic curvature,
defined by the Pfaffian form. This latter fact is stated in form of an integral
identity for the Brouwer degree of the Gauss map, that is classical for $C^2$
functions, but new for $n>2$ in the present context of low regularity.
\end{abstract}

\section{Introduction}

\subsection{Statement of results}
Let $M$ be a compact $n$-dimensional Riemannian manifold, where $n$ is even. 
We adapt Cartan's method of moving frames. Let $X_i$, $i=1,\dots,n$ be an
orthonormal frame on $M$, and let $\theta^i$, $\omega_i^j$, $\Omega_i^j$, $i,j=1,\dots,n$ be
the associated dual forms, connection one-forms and curvature two-forms respectively,
defined by the equations

\begin{equation}
\begin{split}
  \theta^i(X_j)=&\delta^i_j\\
  \d\theta^i=&\sum_{i=1}^n \omega^i_j\wedge\theta^j\\
  \Omega^i_j=&\d\omega^i_j+\sum_{k=1}^n\omega^i_k\wedge\omega_j^k\,,
\end{split}\label{eq:10}
\end{equation}
where $\delta^i_j$ denotes the Kronecker delta, and $i,j\in \{1,\dots,n\}$.
We define the Pfaffian of $(M,g)$ by
\[
  \mathrm{Pf}(\Omega)=\frac{1}{n (n/2)!}\sum_{\zeta\in \Sym(n)}\Omega_{\zeta(1)}^{\zeta(2)}\wedge\dots\wedge \Omega_{\zeta(n-1)}^{\zeta(n)}\,,
\]
where $\Sym(n)$ denotes the group of permutations of $\{1,\dots,n\}$.
It turns out (see \cite{MR532834}) that this formula is independent of the
chosen orthonormal frame, and thus makes $\mathrm{Pf}(\Omega)$ defined on all of
$M$. Additionally, for every isometric immersion $y\in C^2(M;\R^{n+1})$
with normal $\nu:M\to S^n$, we have by Gauss' equation
\begin{equation}
\mathrm{Pf}(\Omega)= \nu^*\sigma_{S^n}\,,\label{eq:7}
\end{equation}
where $\sigma_{S^n}$ denotes the canonical volume element on $S^n$, and $\nu^*$
the pull-back by $\nu$. (Such a
relation  only exists for even $n$, which is the reason why our analysis is limited to this
case.) 
Let us consider the isometric immersion $y$, its normal $\nu$ and the Pfaffian
$\Pf$  in a chart
$U\subset \R^n$. In this chart, the metric $g$ is given by the $n\times n$
matrix-valued function $Dy^TDy$.
As a direct consequence of \eqref{eq:7}, we have the change of variables type
formula
\begin{equation}
\int_{U}\varphi\circ \nu
\Pf=\int_{S^{n}}\varphi(z)\deg(\nu,U,z)\d\H^n(z)\quad\text{ for every
}\varphi\in L^\infty(S^n\setminus \nu(\partial U))\,,\label{eq:8}
\end{equation}
where $\deg(\nu,U,\cdot)$ denotes the Brouwer degree of $\nu:U\to S^n$.
Our first result is the validity of this formula for every isometric immersion  $y\in
C^{1,\alpha}(M;\R^{n+1})$ with $\alpha>n/(n+1)$: 
\begin{theorem}
\label{thm:cov} 
 Let $n$ be even,  $U\subset \R^n$ open and bounded, and $n/(n+1)<\alpha\leq
 1$. Furthermore, let $y\in C^{1,\alpha}(U;\R^{n+1})$ with $g=Dy^TDy\in
 C^\infty(U;\Sym_n^+)$, and let $\Omega$ be the curvature form associated to the
 metric $g$. Then
\[
\int_{U}\varphi\circ \nu
\Pf=\int_{S^{n}}\varphi(z)\deg(\nu,U,z)\d\H^n(z)\quad \text{ for every
}\varphi\in L^\infty(S^n\setminus \nu(\partial U))\,.
\]  
\end{theorem}
Now let us consider  Riemannian manifolds $(M,g)$ whose Pfaffian is a positive multiple of the
volume form. Then for any smooth immersion $y$, the index of the normal map
$\nu$ is positive everywhere. Hence, the Brouwer degree $\deg(\nu,M,\cdot)$ is
positive everywhere on $\nu(M)\setminus\nu(\partial M)$, and one can estimate
the $n$-dimensional Hausdorff measure of images of the normal map by the formula
\eqref{eq:8}. This allows for estimates on \emph{extrinsic curvature}, that we
define following Pogorelov \cite{MR0346714}:

\begin{definition}
Let $y:M\to \R^{n+1}$ be an immersed manifold of class $C^1$.  Denote
the surface normal by $\nu:M\to S^n$. The \emph{extrinsic curvature} of $M$ is
given by
\[
\begin{split}
\sup\Big\{&\sum_{i=1}^N\H^n(\nu(E_i)):\,N\in\N,\,\{E_i\}_{i=1,\dots,N}\\
&  \text{ a  collection of closed disjoint  subsets of
}M\Big\}\in[0,\infty].
\end{split}
\]
If this quantity is finite, we say $y(M)$ is of  bounded extrinsic curvature.
\end{definition}
Here, $\H^n$ denotes $n$-dimensional Hausdorff measure.
Using Theorem \ref{thm:cov}, we can show that the  reasoning above still
applies for isometric immersions $y\in C^{1,\alpha}(U;\R^{n+1})$ if $\alpha>n/(n+1)$:
\begin{theorem}
\label{thm:boundcurv}
  Let $n$ be even, and let $(M,g)$ be a precompact $n$-dimensional Riemannian manifold
  with smooth metric and positive Pfaffian. Furthermore, let $n/(n+1)<\alpha\leq
  1$, and let $y\in C^{1,\alpha}(M;\R^{n+1})$ be an isometric immersion. Then the
  immersed surface $y(M)$ has bounded extrinsic curvature.
\end{theorem}
\subsection{Scientific context}
The Weyl problem is the task of finding an isometric immersion $y\in C^2(S^2; \R^3)$
for a manifold $(S^2,g)$ with positive Gauss curvature (where for simplicity, we assume $g\in C^\infty$). Existence of such an immersion has been proved
independently by Pogorelov and Nirenberg. If a solution of this problem is
unique  up to rigid motions, then it is called \emph{rigid}. The proof of
rigidity for the case of analytic immersions is due to Cohn-Vossen \cite{chon1927zwei}, and for
$C^2$ immersions, it has been given by Pogorelov.\\
\\
The regularity assumption in the Weyl problem is crucial for uniqueness
questions. Recall that an immersion $y\in (S^2;\R^3)$ is called
\emph{short} if every (Lipschitz) path $\gamma\subset M$ gets mapped to a
shorter path $y(\gamma)\subset\R^3$. The famous Nash-Kuiper Theorem
\cite{MR0065993,MR0075640} states that any short immersion can be
approximated in $C^0$ by isometric immersions of regularity $C^1$. (To avoid
confusion, we remark
that the Nash-Kuiper Theorem is not limited to $M=S^2$ and immersions
$M\to\R^3$, but it holds for any short immersion of codimension at
least one.)  Hence, there
exists a vast set of solutions to the Weyl problem in the class of $C^1$
immersions. Historically, this was the first instance of the so-called
``$h$-principle'', nowadays associated with Gromov \cite{MR864505}.\\
\\
Comparing these two cases, it immediately arises the question of what can be said about the intermediate
ones: For which range of $\alpha\in(0,1)$ are isometric immersions $y\in
C^{1,\alpha}(S^2;\R^3)$ rigid? For which range does the $h$-principle hold?   
This is a matter of some interest --  the question of existence of a critical $\alpha$  can be
found as problem 27 in  Yau's list of open problems in geometry
\cite{MR1754714}. For large codimension, a solution has been given by
K\"all\'en in \cite{MR499136}. In this case, the $h$-principle holds for the
whole range $0<\alpha<1$.\\
\\
In codimension one, a partial answer can be found in a series of articles by Borisov 
\cite{MR0104277,MR0104278,MR0116295,MR0104279,MR0131225,MR0192449,MR2047871}. He
proved  that the $h$-principle holds locally for $\alpha<\frac{1}{n^2+n+1}$, where $n$
is the dimension of the manifold (provided the metric
$g$ is analytic), while for $\alpha>\frac23$, $C^{1,\alpha}$-isometric immersions of manifolds $(S^2,g)$
with positive Gauss curvature are rigid.\\
\\
In \cite{conti2012h}, Conti, De Lellis and Sz\'ekelyhidi have given  simplified versions of Borisov's proofs of these
facts. For the case $\alpha<\frac{1}{n^2+n+1}$, it has been shown there
that the $h$-principle also holds for non-analytic $g$. In the recent paper
\cite{2015arXiv151001934D}, it has been proved that in dimension $n=2$, it holds
in the (larger) range $\alpha>\frac15$.
Concerning rigidity, \cite{conti2012h} contains the  statements of our Theorems
\ref{thm:cov} and \ref{thm:boundcurv} for the case $n=2$. Combining the latter with classical
results by Pogorelov on surfaces of bounded extrinsic curvature
\cite{MR0346714}, the rigidity result for isometric immersions $y\in
C^{1,\alpha}(S^2;\R^3)$ with $\alpha>\frac23$ follows (where, of course, Gauss
curvature is assumed to be positive).\\
\\
Our Theorems \ref{thm:cov} and \ref{thm:boundcurv} generalize the results on
extrinsic curvature from
\cite{conti2012h} to even dimension $n>2$. However, the results by
Pogorelov from \cite{MR0346714} that allow to conclude that isometric immersions of bounded extrinsic
curvature  are rigid  have only been proved for the case $n=2$. The question
whether or not their analogues  in higher dimension are valid will not be addressed
here.  Thus, the question whether codimension one $C^{1,\alpha}$-isometric
embeddings of $n$-dimensional manifolds whose Pfaffian is positive for
$\alpha>n/(n+1)$ and even $n$ are rigid, remains open too. \\
\\
As in two dimensions, we require the Pfaffian form to be a positive multiple of
the volume form. In other words, we require the
Gauss-Bonnet integrand
(also known as the Lipschitz-Killing curvature, or Gauss-Kronecker curvature) to
be positive. In passing, we mention that it is a known fact that for  smooth immersions,
positivity of the Gauss-Bonnet integrand implies that the immersed surface is
the boundary of a convex body also in even dimensions $n>2$, see
\cite{MR1871243}.\\
\\
The present paper builds on the recent results by the second author from
\cite{olbermann2015integrability}. The main
focus of that work are the integrability properties of the Brouwer degree. The results
there are achieved by a suitable definition of the distributional Jacobian $\det Du$ for
$u\in C^{0,\alpha}(U;\R^n)$ through real interpolation. In a similar way, the
distributional Jacobian  had been defined by Brezis and Nguyen in
\cite{MR2810795}, building on an idea by Bourgain, Brezis and Mironescu
\cite{MR1781527,MR2075883}.  Here, we will adapt these techniques to the
Pfaffian form. 
Another
ingredient that has been used in \cite{olbermann2015integrability} and will also be
used here, is a well defined notion of integration of H\"older continuous forms
over fractals. 
This follows closely the definitions from the paper
\cite{MR1119189} by Harrison and Norton.
The main ingredients of our proofs will be suitable generalizations of the results from
\cite{olbermann2015integrability}; for the convenience of the reader, we will
give full proofs of these statements, and not refer to that work. 

\subsection{Plan of the paper}
In Section \ref{sec:results-from-liter}, we will recall some well known facts:
The Gauss-Bonnet-Chern Theorem, real interpolation by the trace method, the
definition of box dimension, the Whitney decomposition of an open subset of
$\R^n$, properties of the level sets of H\"older functions and the approximation
of $C^{1,\alpha}$ isometric immersions by mollification. These results will be
used in Section \ref{sec:distr-pfaff-jacob} to give a suitable definition of the
distributional Jacobian, the distributional Pfaffian and well defined notions of
their integrals over sets with fractal boundary. This section parallels most of
the ideas from \cite{olbermann2015integrability} and adapts them to the current
setting. In Section
\ref{sec:proof-mainthm}, we combine the results from Section
\ref{sec:distr-pfaff-jacob} to
prove Theorem \ref{thm:cov}. Theorem \ref{thm:boundcurv} will then be obtained from
Theorem \ref{thm:cov} by arguing in exactly the same way as in the proof of the
case $n=2$ in \cite{conti2012h}.
\subsection{Notation}
Except for the proof of Theorem \ref{thm:boundcurv}, our investigations will
take place in a single chart $U\subset\R^n$ of an $n$-dimensional Riemannian
manifold $(M,g)$. 
For $k=0,1,2,\dots$, the $C^k$-norms $\|\cdot\|_{C^k(U)}$ are defined by
\[
\|u\|_{C^k(U)}=\sum_{0\leq j\leq k}\sup_{x\in U}|D^ju(x)|\,.
\]
For $\alpha\in(0,1]$, the H\"older semi-norms $[\cdot]_\alpha$ are defined by
\[
[u]_\alpha=\sup_{\substack{x,x'\in U\\x\neq y}}\frac{|u(x)-u(x')|}{|x-x'|}\,.
\]
Finally, the $C^{k,\alpha}$ norms
$\|\cdot\|_{C^{k,\alpha}(U)}$ are defined by
\[
\|u\|_{C^{k,\alpha}(U)}=\|u\|_{C^k(U)}+[D^ku]_\alpha\,.
\]
In the chart $U$, the metric $g$ is a smooth function on $U$ with values in
the positive definite $n\times n$ matrices $\Sym_n^+$, $g\in
C^\infty(U;\Sym^+_n)$. We make $C^{k,\alpha}(U;\Sym^+_n)$ a normed space by
setting
\[
\|g\|_{C^{k,\alpha}(U;\Sym^+_n)}=\sum_{i,j=1,\dots,n}
  \|g_{ij}\|_{C^{k,\alpha}(U)}\,.
\]

An immersion $y\in C^1(U;\R^{n+1})$ is an isometric
immersion (w.r.t.~$g$) if and only if $Dy^TDy=g$.
The tangent space at every $x\in U$ will be identified with
$\R^n$, and hence vector fields are identified with $\R^n$ valued functions on
$U$. Let $\Lambda^{p}\R^n$ denote the set of rank $p$ multi-vectors in $\R^n$, i.e.,
the linear space 
\[
\Lambda^p\R^n=\left\{ \sum_{i_1,\dots,i_p\in\{1,\dots,n\}}a_{i_1,\dots,i_p}\d x_{i_1}\wedge\dots \wedge \d x_{i_p} :\,
 a_{i_1,\dots,i_p}\in\R\right\}\,,
\]
Hence, $p$-forms will be functions on $U$ with values in $\Lambda^p\R^n$. We 
make $C^{k,\alpha}(U;\Lambda^p\R^n)$  a normed space by setting
\[
\|a\|_{C^{k,\alpha}(U;\Lambda^p\R^n)}=\sum_{i_1,\dots,i_p}\|a_{i_1,\dots,i_p}\|_{C^{k,\alpha}(U)}
\]
for $a=\sum_{i_1,\dots,i_p}a_{i_1,\dots,i_p}\d x_{i_1}\wedge\dots \wedge \d x_{i_p} $.\\
\\
The symbol ``$C$'' will be used as follows: A statement such as ``$a\leq C(\alpha) b$''
is to be understood as ``there exists a numerical constant $C$ only depending on
$\alpha$ such that $a\leq
Cg$''. Whenever the dependence of $C$ on other parameters is clear, we also
write ``$a\lesssim b$'' in such a situation. \\
\\
The non-negative real line will be denoted by $\R^+=[0,\infty)$. On $\R^+$, we
write $\d t/t$ for the measure defined by $A\mapsto \int_{\R^+}\chi_A(t) \d
t/t$, where $\chi_A$ is the characteristic function of the
measurable set $A\subset \R^+$.

\subsection*{Acknowledgments} This paper presents the main results of the first author's Masters thesis \cite{thesisBehr}.
\section{Results from the literature}
\label{sec:results-from-liter}
\subsection{The Gauss-Bonnet-Chern Theorem}
The Gauss-Bonnet-Chern Theorem states in particular that the Pfaffian form is an exact form, with an explicit formula for a primitive that can be written as a polynomial in the connection and curvature forms.\\
Let $(M,g)$ be a Riemannian manifold of even dimension $n$. 
For a given orthonormal frame $\{X_i\}_{i=1,\dots,n}$ and associated connection and curvature forms $\omega_i^j$, $\Omega_i^j$, let the forms $\Phi_i$, $i=1,\dots,n/2-1$ be defined by

\begin{equation}
\begin{split}
  \Phi_i=\sum_{\substack{\zeta\in\Sym(n)\\ \zeta(1)=1}}&\sgn\, \zeta\,
  \omega^1_{\zeta(2)}\wedge
  \omega^1_{\zeta(3)}\wedge\dots\wedge\omega^1_{\zeta(n-2i)}
\wedge
  \Omega^{\zeta(n-2i+1)}_{\zeta(n-2i+2)}\wedge\dots\wedge
  \Omega^{\zeta(n-1)}_{\zeta(n)}\,.
\end{split}\label{eq:3}
\end{equation}
\begin{theorem}[Gauss-Bonnet-Chern, \cite{MR0011027}]
\label{thm:GBC}
We have $\Pf=\d\Pi(\omega)$, where 

\begin{equation}
\Pi(\omega)=\frac{1}{\pi^n}\sum_{i=0}^{n-1}\frac{(-1)^i}{1\cdot 3\cdots (2n-2i-1)i!2^{n+i}}\Phi_i\,,\label{eq:11}
\end{equation}
with $\Phi_i$, $i=1,\dots,n/2-1$, as defined in \eqref{eq:3}, and $\Omega^i_j$
defined as a function of the $\omega^k_l$ through \eqref{eq:10}.
\end{theorem}

\subsection{Real interpolation via  the trace method}
\label{sec:real-interp-via}
We are going to use some standard constructions from real interpolation theory.
The following way to introduce the real interpolation spaces is due to Lions \cite{MR0159212}.
Let  $E_0$ and $E_1$ be two Banach spaces that are continuously embedded in a 
Hausdorff topological vector space. This is only necessary to guarantee that $E_0 \cap E_1$ and 
$E_0 + E_1 = \{ e_0 + e_1 \mid e_0 \in E_0, \ e_1 \in E_1 \}$ are also Banach spaces with the following norms:
\[
\begin{split}
\|x\|_{E_0 \cap E_1} &:= \max \{\|x\|_{E_0}, \|x\|_{E_1} \} \\
\|x\|_{E_0 + E_1} &:= \inf \{\|x_0\|_{E_0} + \|x_1\|_{E_1} \ : \ x_0 \in E_0, \ x_1 \in E_1, \ x_0 + x_1 = x \}
\end{split}
\]
\begin{definition}
For $\theta \in (0,1) $ and  $1 \leq p \leq \infty$ we denote by $V(p, \theta,
E_1, E_0)$ the set of all functions $u \in W^{1,p}_{loc}(\R_+,E_0\cap E_1;\d
t/t)$ such that, with $u_{*,\theta}(t):=t^\theta u(t)$ and
$u'_{*,\theta}:=t^\theta u'(t)$, we have
\[
u_{*,\theta} \in  L^p(\R^+, E_1; \d t/t) ,\quad u'_{*,\theta} \in  L^p(\R^+, E_0; \d t/t)\,,
\]
and we define a norm on $V = V(p, \theta, E_1, E_0)$ by
\[
\|u\|_V := \|u_{*,\theta}\|_{ L^p(\R^+, E_1; \d t /t)} + \| u'_{*,\theta}\|_{ L^p(\R^+, E_0; \d t/t)}\,.
\]
\end{definition}
It can be shown that those functions are continuous in $t=0$. We define the real interpolation spaces as follows:
\begin{definition} The real interpolation space $(E_0, E_1)_ {\theta, p}$ 
is defined as set of traces of functions  belonging to $V(p, 1-\theta, E_1, E_0)$ at $t=0$ together with the norm:
\[
\|x\|^{Tr}_{(\theta, p)} = \inf \{ \|u\|_V \mid u \in V(p,1-\theta,E_1,E_0), \, u(0) = x \}
\]
\end{definition}
We conclude with two estimates for $x \in E_{0} \cap E_{1}$.
\begin{lemma} 
\label{lem:interpolest}
Let $x \in E_0 \cap E_1$. Then
\begin{equation} \label{eq:interpolest1}
\|x\|^{Tr}_{(\theta, \infty)} \leq C \|x\|_{E_0}^{1-\theta} \|x\|_{E_1}^{\theta}\,,
\end{equation}
and for $u \in V(\infty,1-\theta,E_1,E_0)$ with $u(0) = x$:
\begin{equation} \label{eq:interpolest2}
\|x-u(t)\|_{E_{0}} \leq C(\theta) t^{\theta} \|u\|_{V}\,.
\end{equation}
\end{lemma}
\begin{proof}
For $r > 0$, we set  
\[
u(t) := \begin{cases} (1 - t/r)x \quad &\text{if } t<r \\ 0 \quad &\text{else} \end{cases} \, .
\]
Note that
\[
\|t^{1-\theta} u(t)\|_{E_1} \leq r^{1-\theta} \|x\|_{E_1}
\]
and
\[
\|t^{1-\theta} u'(t)\|_{E_0} \leq r^{-\theta} \|x\|_{E_0} 
\]
and choose $r = \frac{\|x\|_{E_0}}{\|x\|_{E_1}}$. This yields \eqref{eq:interpolest1}.\\
\\
For the second estimate, we write $u(0) - u(t) = \int_{0}^{t} u'(s) \d s$ and conclude
\[
\begin{split}
	\|x -u(t)\|_{E_{0}} &\leq \int_{0}^{t} \|u'(s)\|_{E_{0}} \d s\\
	&\leq \int_{0}^{t} s^{\theta-1} \|s^{1-\theta} u'(s)\|_{E_{0}} \d s \\
	&\leq \int_{0}^{t} s^{\theta-1} \|u\|_{V} \d s \\
	&\leq C(\theta)t^{\theta} \|u\|_{V} \, .
\end{split}
\]
\end{proof}

\subsection{Box dimension, Whitney decomposition and level sets of H{\"o}lder  functions}
We recall the following decomposition of an open set into cubes, due to Whitney \cite{MR1501735}, and
a bound of the number of cubes of a certain size.  
\begin{lemma} Let $U\subset \R^n$ be open. Then there is a countable collection of mutually disjoint axis-aligned cubes $W$ whose diameters are comparable to their distance from $\partial U$, i.e.
\begin{enumerate}[(i)]
\item $U = \cup_{Q\in W} \overline{Q}$
\item For all $Q \in W$ there is $k \in \Z$ and $l \in \Z^n$ such that $Q = 2^{-k}(l + {(0,1)}^n)$. We denote the sub-collection of all cubes of size $k$ by $W_k$.
\item $Q \cap Q' = \varnothing$ if $Q \not = Q'$
\item $\diam Q \leq \dist (Q, \partial U) \leq 4 \ \diam Q$
\end{enumerate}
\end{lemma}
For a proof we refer to \cite{MR0290095}. 
\\
It turns out that $|W_k|$  is related to the \emph{box dimension} of the
boundary of the decomposed set -- this is made precise in Lemma
\ref{lem:whitneybox} below. First, we give the definition of box dimension
(cf.~e.g.~\cite{MR3236784}):
\begin{definition} \label{dfn:boxdim} Let $E \subset \R^{n}$ be bounded and $\beta>0$. We introduce the (upper) Hausdorff-type content
\begin{equation}
\label{eq:20}
\overline {H}^{\beta}(E) = \limsup_{\e \rightarrow 0} \inf \{ m\e^{\beta} \mid E \subset \cup_{i=1}^{m} B_{\e}(x_{i}) \}
\end{equation}
and define the (upper) box dimension of $E$ to be
\begin{equation*}
\bd E = \sup \{ s \mid \overline{H}^{s}(E)  = \infty \}\,.
\end{equation*}
\end{definition}
\begin{lemma} [\cite{MR880256}, Theorem 3.12]\label{lem:whitneybox} Let $U \subset \R^n$ be open and bounded such that $\bd \partial U < d$. Then there is $M>0$ such that
\begin{equation*}
|W_k| \leq M 2^{d k}\,.
\end{equation*}
\end{lemma}	
 
 We conclude this subsection with a lemma regarding the box dimension of pre-images of H{\"o}lder continuous functions.
\begin{lemma} \label{lem:hoelderpreimboxdim} Let $U \subset \R^n$ be open and
  bounded, $f \in C^{0,\alpha}(U)$. If $\beta \geq n - \alpha$
  then we have 
\[
\bd f^{-1}(r)  \leq \beta\quad\text{ for a.e. }r\in \R\,.
\]
\end{lemma}
\begin{proof}
For the sake of contradiction, assume that
there exists a set $A\subset \R$ of positive measure such that for all $r \in A$,
\[
\bd f^{-1}(r)  > \beta \, .
\]
Let $k \in \N$ be arbitrary. Then, by assumption, there is $0 < \e_r < 1$ for
every $r \in A$
such that 
\begin{equation} \label{eq:numberofballs}
\inf \{ m \e_r^\beta \mid f^{-1}(r) \subset \cup_{i=1}^m B_{\e_r}(x_i) \} > k \, .
\end{equation}
We conclude that at least $k \e_r^{-\beta}$  balls of radius $\e_r$ are
necessary to cover $f^{-1}(r)$ for $r\in A$.
Obviously, 
\[
A \subset \cup_{r \in A} \overline{B}_{C\e_r^\alpha}(r) \, .
\]
Using the Vitali covering lemma, we obtain $\mathcal{J} \subset A$ countable such that the balls $B_{C\e_j^\alpha}(j)$ 
(for $j \in \mathcal{J}$) are pairwise disjoint and $A \subset \cup_{j\in \mathcal{J}} B_{5C\e_j^\alpha}(j)$. 
Therefore, we have that
\[
\sum_{j \in \mathcal{J}} 2 C \e_j^\alpha \geq \frac{1}{5} \L^1(A) \, .
\]
Again using the Vitali covering lemma on the families $\{  \overline{B}_{\frac{1}{10}\e_j} (x) \mid x \in f^{-1}(j) \}$,
we obtain countable collections $\mathcal{I}_j$ of pairwise disjoint balls, such that
\[
f^{-1}(j) \subset \bigcup_{B \in \mathcal{I}_j} \hat{B} \, ,
\]
where $\hat{B}$ denotes the concentric ball with five times the radius. Note
that by \eqref{eq:numberofballs}, we have $|\mathcal{I}_j| > k \e^{-\beta}$ for
every $j\in \mathcal J$.
Moreover, for $ B = \overline{B}_{\frac{1}{10}\e_j} (x) \in \mathcal{I}_j$ and $B' = \overline{B}_{\frac{1}{10}\e_{j'} }(x') \in \mathcal{I}_{j'}$ we observe that if $B \cap B' \neq \varnothing$:
\[
|j - j'| = |f(x) - f(x')| \leq \|f\|_{C^{0,\alpha}} |x-x'|^\alpha <  \|f\|_{C^{0,\alpha}} \e_j^\alpha
\]
and hence $j=j'$. We conclude that 
\[
\begin{split}
\L^n(f^{-1}(A)) &\geq \sum_{j \in \mathcal{J}} \sum_{B \in \mathcal{I}_j} \L^n(B)  \\
&\gtrsim \sum_{j \in \mathcal{J}} k \e_j^{n-\beta} \geq \sum_{j \in \mathcal{J}} k \e_j^{\alpha} \geq \frac{k}{10\|f\|_{C^{0,\alpha}}} \L^1(A) \, ,
\end{split}
\]
where we used that by assumption $n - \beta \leq \alpha$. This is a contradiction since $k$ was arbitrary.
\end{proof}

\subsection{Approximating $C^{1,\alpha}$ isometric immersions by smooth ones}
\label{sec:appr-c1-alpha}
Let $U\subset\R^n$ be open and bounded, and $\alpha\in(0,1]$. We consider an 
immersion $y\in C^{1, \alpha}(U;\R^{n+1})$, and  write
\[
Dy^TDy=g\,.
\]
We will assume that $g$ is a smooth function $U\to \R^{n\times n}$, with values
in the positive definite $n\times n$ matrices $\Sym^+_n$. I.e., 
$g\in C^\infty(U;\mathrm{Sym}^+_n)$.\\  
\\
Let $\varphi$ be a standard symmetric mollifier; i.e., $\varphi \in
C^{\infty}_{c}(\R^{n})$, $\varphi(x)=\varphi(-x)$ for $x\in \R^n$ and
$\int_\R^n \varphi(x)\d x = 1$.
We set $\varphi_{\e}(x) = \e^{-n}\varphi(\frac{x}{\e})$ and define the
mollifications $y_\e$ as well as their induced metrics $g_\e$ by setting
\begin{equation}
\begin{split}
  y_\e:=&\varphi_\e*(y\chi_U)\\
  g_\e:=&Dy_\e^TDy_\e\,.
\end{split}\label{eq:16}
\end{equation}

The following estimate from \cite{conti2012h} is crucial for our analysis: 
\begin{lemma}[Proposition 1 in \cite{conti2012h}] \label{lem:quadratic}
 If $y\in C^{1,\alpha}$, then we have
\[
\|Dy^{T}_\e Dy_\e-Dy^TDy\|_{C^r(U;\Sym^+_n)}\leq C\e^{2\alpha-r}\,.
\]
\end{lemma}
For  the reader who is unfamiliar with this estimate, we mention that its proof
is based on the  commutator estimate 
\[
\|(fg) \ast \varphi_{\e} - ( f \ast \varphi_{\e})( g \ast \varphi_{\e})\|_{C^{r}} \leq C_{r} \e^{2\alpha -r} [f]_{C^{0,\alpha}} [g]_{C^{0,\alpha}} \, .
\]
which  appeared first in context of the Onsager conjecture on Energy Conservation
for Euler's equation (see \cite{MR1298949}).\\
\\
From   Lemma \ref{lem:quadratic} and the interpolation inequality
$\|u\|_{C^{1,\beta}}\leq C \|u\|_{C^1}^{1-\beta}\|u\|_{C^2}^\beta$ (see Lemma \ref{lem:interpolest}), we obtain
\begin{lemma}
\label{lem:gcon}
If $y\in C^{1,\alpha}(U,\R^n)$ and $g_\e$ is defined by \eqref{eq:16}, then 
\[
g_\e\to g \quad\text{ in } C^{1, \beta}(U;\Sym^+_n)\quad \text{ for all
}\beta<2\alpha-1\,.
\]
\end{lemma}

For the rest of this section, let $0<\beta<2\alpha-1$.
Let $\mathcal U$ be a small neighborhood of $g$ in $C^{1,\beta}(U,\Sym_n^+)$.
At every point $x\in U$, an ordered basis of the tangent space is given by
$(\partial_1,\dots,\partial_n)$. To this ordered basis, and for every $\tilde
g\in\mathcal U$, we may apply the
Gram-Schmidt process with respect to $\tilde g(x)$, and obtain an orthonormal frame
$(X_1,\dots,X_n)$. Let $U_1$ be a suitably chosen
 neighborhood of $g(x)$ in $\Sym_n^+$. Note that the  map 
\[
\begin{split}
  U_1&\to
  \R^{n\times n}\\
  \tilde g(x)&\mapsto (X_1(x),\dots,X_n(x))
\end{split}
\]
 is in $C^\infty(U_1;\R^{n\times n})$.  Hence, the map
\[
\begin{split}
  \mathcal U\to& C^{1,\beta}(U;\R^{n\times n})\\
  \tilde g\mapsto & (X_1,\dots,X_n)
\end{split}
\]
is continuous. (It is in order to have this continuity
why we choose a particular orthonormal frame, instead of choosing an arbitrary one.)
We define dual one-forms $\theta^i\in C^\infty(U;\Lambda^1\R^n)$ by requiring
$\theta^i(X_j)=\delta_j^i$, where $\delta_j^i$ is the Kronecker delta. By this
definition, we have 
  \[
(\theta^i(\partial_j))_{i,j=1,\dots,n}=\left((\d
  x_j(X_i))_{i,j=1,\dots,n}\right)^{-1}
\]
which is again smooth as a function of the $X_i$ (in some uniform neighborhood
of the $X_i(x)$ given by $\tilde g(x)\mapsto (X_1(x),\dots,X_n(x))$, $x\in
U$, $\tilde g\in\mathcal U$), 
and hence the map $\tilde g\mapsto (\theta^1,\dots,\theta^n)$ is continuous from
$\mathcal U$ to $C^{1,\beta}(U;\Lambda^1\R^n)^n$.  The connection one forms associated
to the frame $X$ are defined by the first structural equation,
\[
\d\theta^i=\sum_j\omega^i_j\wedge \theta^j\,,\quad \omega^i_j=-\omega_i^j\,.
\]
This defines $(\omega_i^j(x))_{i,j=1,\dots,n}$ as a smooth function of
$(\d\theta^1(x),\dots,\d\theta(x))$ and $(\theta^1(x),\dots,\theta(x))$. This
in turn implies our final conclusion, which we state as a lemma:
\begin{lemma}
\label{lem:contlem}
The map $\tilde g\mapsto (\omega_i^j)_{i,j\in 1,\dots,n}$ defined above, associating to a
metric $\tilde g$ a connection-one form satisfying \eqref{eq:10}, is continuous from
  some neighborhood $\mathcal U\subset C^{1,\beta}(U;\Sym^+_n)$ of $g$ to $C^{0,\beta}(U;\Lambda^1\R^n)^{n\times n}$.
\end{lemma}


\section{Distributional Pfaffians,  Jacobians, and their integrals over sets
  with fractal boundary}
\label{sec:distr-pfaff-jacob}
\subsection{Distributional  Jacobians and Pfaffians for H\"older functions}
\label{sec:weak-pfaff-jacob}
The aim of the present subsection is to give a definition of the Jacobian
determinant for H\"older continuous functions, as well as a definition of the
Pfaffian form for metrics with H\"older continuous derivatives. In both cases,
the H\"older exponent has to be large enough for this to be possible. The main
technical ingredient is real interpolation, and the core of the argument is
contained in the proof of Proposition \ref{prop:traceinter} below.\\
\\
Recall that for an open and bounded set $U \subset \R^n$, we denote the space of
smooth $k$-forms on $U$ by $C^{\infty}(U;\Lambda^{k}\R^n)$. Furthermore, we set
\[
C^{\infty}_{cl}(U;\Lambda^{k}\R^n) = \{ \omega \in C^{\infty}(U;\Lambda^{k}\R^n) \mid \d \omega = 0 \} \, . 
\]
In the following, we will only be interested in the case $k=n-1$. We introduce two norms on $C^{\infty}(U;\Lambda^{n-1}\R^n) / C^{\infty}_{cl}(U;\Lambda^{n-1}\R^n)$,
\[
\begin{split}
\|\omega\|_{X^{n-1}_0} & := \inf \{ \|\omega + a\|_{C^{0}(U;\Lambda^{{n-1}}\R^n)} \mid a \in C^{\infty}_{cl}(U;\Lambda^{{n-1}}\R^n) \} \\
\|\omega\|_{X^{n-1}_1} & := \|\d \omega\|_{C^{0}(U;\Lambda^{{n}}\R^n)}
\end{split}
\]
and denote by $X^{n-1}_0$ and $X^{n-1}_1$ the completion of $C^{\infty}(U;\Lambda^{{n-1}}\R^n) / C^{\infty}_{cl}(U;\Lambda^{{n-1}}\R^n)$ with respect to these norms.
We introduce the shorthand notation
\[ 
X_\theta=(X_0^{n-1},X_1^{n-1})_{\theta,\infty} \, .
\]
\begin{proposition}
\label{prop:traceinter}
  Let $k_1,\dots,k_J\in \N$ and $0\leq I<J$ with $\sum_{i=1}^Jk_i+I=n-1$. For
  $\omega_{i}\in C^\infty(U;\Lambda^{k_i}\R^n)$ with  $i=1,\dots,J$,  set
\[
M_{k,I}(\omega_1,\dots,\omega_J)=\omega_1\wedge\dots\wedge\omega_{J-I}\wedge \d \omega_{J-I+1}\wedge\dots\wedge \d \omega_{J}\,.
\]
Furthermore, let $\beta\in (0,1]^J$ such that $\theta:= \min_{i=1, \dots ,J - I} \beta_i + \sum_{i=J-I+1}^J \beta_i - I \in (0,1)$, and $X_\theta=(X_0^{n-1},X_1^{n-1})_{\theta,\infty}$.
Then 
\begin{equation}
\|M_{k,I}(\omega_1,\dots,\omega_J)\|_{X_\theta}\leq \prod_{i=1}^J\|\omega\|_{C^{0,\beta_i}}\,.\label{eq:17}
\end{equation}
Moreover, for $\tilde \theta<\theta$, $M_{k,I}$  extends to a multi-linear continuous  operator
\[
C^{0,\beta_1}(U;\Lambda^{k_1}\R^n)\times\dots\times
C^{0,\beta_J}(U;\Lambda^{k_J}\R^n)\to X_{\tilde \theta}\,.
\]
\end{proposition}
\begin{proof}
  We use that $C^{0,\beta_i} = (C^0,C^1)_{\beta_i,\infty}$, write
  \[
  \omega_i = \sum_{1 \leq i_1 < \dots < i_{k_i} \leq n} v^i_{i_1, \dots, i_{k_i}}(0) \d x_{i_1} \wedge \dots \wedge \d x_{i_{k_i}}
  \]
  for some $v^i_{i_1, \dots , i_{k_i}} \in V(\infty, 1-\beta_i, C^1, C^0)$. Without loss of generality, we may assume that $\supp v^i_{i_1, \dots, i_{k_i}} \subset [0,1]$ and set
  \[
  v^i(t) = \sum_{1 \leq i_1 < \dots < i_{k_i}\leq n } v^i_{i_1, \dots, i_{k_i} }(t) \d x_{i_1} \wedge \dots \wedge \d x_{i_{k_i}} \, .
  \]
  The main idea is to write $\d M_{k,I} (v_1(t), \dots , v_J(t))$  and $(M_{k,I} (v_1(t), \dots , v_J(t)))'$ as wedge products of $(I+1)$ derivatives $\d v^i$ or $(v^i)'$ and $J-I-1$ factors of the form $v^i$.
  Observe that
  \[
  \d M_{k,I}(\omega_1, \dots , \omega_J) = \sum_{i=1}^{J-I} (-1)^{\sum_{j=1}^{i-1} k_j} \omega_1 \wedge \dots \wedge \omega_{i-1} \wedge \d \omega_i \wedge \omega_{i+1} \wedge \dots \wedge \omega_{J-I} \wedge \d \omega_{J-I+1} \wedge \dots \wedge \d \omega_J
  \]
  and hence, for all $t\in\R^+$,
  \[
  \| M_{k,I} (v^1(t), \dots , v^J(t)) \|_{X^{n-1}_1} \leq \sum_{i=1}^{J-I} \|v^i(t)\|_{C^1} \prod_{\substack{j= 1,\dots, J-I \\ j \neq i}} \|v^j(t)\|_{C^0} \prod_{l=J-I+1,\dots, J} \|v^l(t)\|_{C^1} \ .
  \]
  For $(M_{k,I}(v^1, \dots , v^J))'$, we use that the $X^{n-1}_0$-norm is only defined up to a closed form to avoid factors involving two derivatives. Note that
  \[
  \begin{split}
    v^1 \wedge &\dots \wedge v^{J-I} \wedge \d v^{J-I+1} \wedge \dots \wedge \d v^{i-1} \wedge (\d v^i)' \wedge \d v^{i+1} \wedge \dots \wedge \d v^J  \\
    = &\pm \d \left(v^1 \wedge \dots \wedge v^{J-I} \right) \wedge \d v^{J-I+1}
    \wedge \dots \wedge \d v^{i-1} \wedge  (v^i)' \wedge \d v^{i+1} \wedge \dots
    \wedge \d v^J\\
 &\pm \d \left(v^1 \wedge \dots \wedge v^{J-I} \wedge (v^i)' \right) \wedge \d v^{J-I+1}
    \wedge \dots \wedge \d v^{i-1}  \wedge \d v^{i+1} \wedge \dots
    \wedge \d v^J\,.
  \end{split}
  \]
Note that the $n-1$ form in the last line above is closed.  
When computing the $X^{n-1}_0$-norm of $(M_{k,I}(v^1, \dots , v^J))'$, we therefore may replace every
term   involving $(\d v^i)'$ with a sum of terms involving $I$ exterior and one
time   derivative. Thus we have
  \[
  \|(M_{k,I}(v^1(t), \dots , v^J(t)) )'\|_{X^{n-1}_0} \leq \sum_{i=1}^{J} \|(v^i)'(t)\|_{C^0} \prod_{\substack{j= 1,\dots, J-I \\ j \neq i}} \|v^j(t)\|_{C^0} \prod_{\substack{l=J-I+1,\dots, J \\ l \neq i}} \|v^l(t)\|_{C^1} \,
  \]
for all $t\in \R^+$.
By the definition of $\theta$,   we have that
  \[
  t^{1-\theta} = t^{1-\min_{i=1,\dots, J-I} \beta_i} \prod_{i=J-I + 1}^J t^{1-\beta_i}\,.
  \]
Recall that
  \[
  t^{1-\beta_i}\left( \|v^i(t)\|_{C^1} + \|(v^i)'(t)\|_{C^0}\right) \leq
  \|v^i\|_{V}\quad\text{ for  all } t\in\R^+\,,
  \]
  and observe that by the second estimate in Lemma \ref{lem:interpolest}, we have
  \[{}
  \|v^i(t)\|_{C^0} \leq \|v^i(0)\|_{C^0} + \|v^i(t) - v^{i}(0)\|_{C^0} \leq
  \|v^i(0)\|_{C^0} + t^{\beta_{i}}\|v^i\|_{V}\lesssim \|v^i\|_{V} \, .
  \]
  We conclude that for all $t\in \R^+$, we have
  \[
  \begin{split}
    \| t^{1-\theta} M_{k,I} (v^1(t), \dots , v^J(t)) \|_{X^{n-1}_1} &\lesssim \prod_{j= 1,\dots, J} \|v^j\|_{V(\infty, 1-\beta_{j}, C^{1},C^{0})} \\
    \| t^{1-\theta} (M_{k,I} (v^1(t), \dots , v^J(t)))' \|_{X^{n-1}_0} &\lesssim \prod_{j= 1,\dots, J} \|v^j\|_{V(\infty, 1-\beta_{j}, C^{1},C^{0})}\,.
  \end{split}
  \]
   Taking appropriate infima on both sides of these estimates completes the proof of
    \eqref{eq:17}.\\
\\
 To prove the statement about the extension of $M_{k,I}$, we
    only need to choose $\tilde\beta_i<\beta_i$ for $i=1,\dots,J$ such that
    \[
\tilde\theta=\min_{i=1,\dots,J-I}\tilde \beta_i+\sum_{i=J-I+1}^J\tilde\beta_i-I
\]
and note that $\omega_i$ can be approximated in
$C^{0,\tilde\beta_i}(U;\Lambda^{k_i}\R^n)$ by smooth functions
$\omega_{i,\delta}\in C^{\infty}(U;\Lambda^{k_i}\R^n)$,
\[
\omega_{i,\delta}\to \omega_i\quad\text{ in } C^{0,\tilde
  \beta_i}(U,\Lambda^{k_i}\R^n) \text{ as } \delta\to 0\,.
\]
By the estimate
\eqref{eq:17} applied with $\tilde \theta$  and $\tilde\beta_i$, $i=1,\dots,J$, 
\[
\delta\mapsto M(\omega_{1,\delta},\dots,\omega_{J,\delta})
\]
is a continuous function with values in $X_{\tilde \theta}$, whose limit does not depend on the
choice of the approximations $\omega_{i,\delta}$. This proves the proposition.
\end{proof}
\begin{remark}
  \label{rem:Pirem}
Let $\omega^i_j$ and $\Omega^i_j$ denote the connection and curvature forms associated
to the orthonormal frame $\{X_i\}_{i=1,\dots,n}$ that we fixed above in Section \ref{sec:appr-c1-alpha}. As a consequence of the
defining equation for $\Omega$, the equation \eqref{eq:3} defines the forms
$\Phi_i$ as a polynomial in the $\d \omega^i_j$ and $\omega^i_j$, and by Theorem
\ref{thm:GBC},  there exist constants $c^I_{i_1,\dots,i_{2(n-I-1)}}\in \R$ for
$I=1,\dots, n/2-1$, $i_1,\dots,i_{2(n-I-1)}\in\{1,\dots,n\}$, such that the Gauss-Bonnet-Chern form $\Pi$ can be written as
\[
\Pi(\omega)=\sum_{\substack{I=1,\dots,n/2-1\\
i_1,\dots,i_{2(n-I-1)}\in\{1,\dots,n\}}}
c^I_{i_1,\dots,i_{2(n-I-1)}}M_{\underbrace{(1,\dots,1)}_{\text{J times }1},I}\left(\omega^{i_1}_{i_2},\dots,\omega^{i_{2(n-I)-3}}_{i_{2(n-I-1)}}\right)\,,
\]
where $J=n-1-I$.
By Proposition \ref{prop:traceinter}, it follows that the map
\[
\omega\mapsto \Pi(\omega)
\]
defined by \eqref{eq:11} is continuous from $C^{0,\beta}(U;\Lambda^1\R^n)^{n\times n}$ to $X_\theta$
for every $\theta<n\beta/2-(n/2-1)$.
\end{remark}
\subsection{Integrating distributional Pfaffians and Jacobians over sets with fractal boundary}

The interpolation space $X_\theta=(X_0^{n-1},X_1^{n-1})_{\theta,\infty}$ has been
chosen in a way such that elements in this space can be integrated over fractals
of dimension up to (but not including) $n-1+\theta$; it will be shown now how
this works.
We adapt the arguments from \cite{MR1119189}, and  give a well
defined meaning to integrals over differentials $\d M$ with $M\in X_\theta$.

We fix some $U\subset \R^n$ with $d:=\bd\partial U<n-1+\theta$. Let $W$ denote the Whitney
decomposition of $U$. 
Recalling the properties of trace spaces from Section \ref{sec:real-interp-via}, we have that
for $M\in X_\theta$  there exists $M(\cdot)\in W^{1,1}_{\mathrm{loc}}(\R^+;C^1(U;\Lambda^{n-1}\R^n))$ such that
\[
t^{1-\theta}\left(\|M(t)\|_{C^1}+\|M'(t)\|_{C^0}\right)\leq
\|M\|_{X_\theta}\quad\text{ for all }t\in\R^+\,,
\]
and 
\[
\lim_{t\to 0}\|M-M(t)\|_{C^0}=0\,.
\]

\begin{definition}
\label{def:fracint}
  For $M\in (X_0^{n-1},X_1^{n-1})$,   
we define the integral $\int_U\d M$ by 
\[
\int_U\d M:=\sum_{Q\in W} \int_Q \d M(\diam Q)+\int_{\partial Q} (M-M(\diam
Q))\,.
\]
\end{definition}

\begin{lemma}
\label{lem:intwelldef}
The above definition makes $\int_U \d M$ well defined for $M\in X_\theta$ for $n-1+\theta>d$. Furthermore, the map
\[
M\mapsto\int_U \d M
\]
is continuous on $X_\theta$. 
\end{lemma}
\begin{proof}
Let us fix $M\in X_\theta$ and choose $M(\cdot)\in W^{1,1}_{\mathrm{loc}}(\R^+;C^1(U;\Lambda^{n-1}\R^n))$ as
above. Let $Q\in W$. We estimate
  \[
  \begin{split}
    \left| \int_Q \d M(\diam Q) \right| \leq &\L^n(Q) \|M(\diam Q)
    \|_{X^{n-1}_1}\\
    \leq &\L^n(Q) (\diam Q)^{\theta - 1} \|M\|_{X_\theta}
  \end{split}
  \]
  and
  \[
  \begin{split}
    \left| \int_{\partial Q} (M-M(\diam Q)) \right| \leq& \H^{n-1}(Q) \|M-M(\diam
    Q)\|_{X^{n-1}_0}\\
 \lesssim &\H^{n-1}(Q)(\diam Q)^\theta \|M\|_{X_\theta} \, ,
  \end{split}
  \]
where we have used Lemma \ref{lem:interpolest} in the second estimate.
  By Lemma \ref{lem:whitneybox}, the number of cubes in $W$ of sidelength
  $2^{-k}$ can be estimated by $C 2^{kd}$, where the constant $C$ may depend on
  the domain $U$, and $d=\bd \partial U$. Hence we may estimate
  \[
  \begin{split}
    \left| \int_U \d M \right| \leq &\L^n(Q) (\diam Q)^{\theta - 1} \|M\|_{X_\theta}+\sum_{Q \in W} \H^{n-1}(Q) (\diam Q)^\theta
    \|M\|_{X_\theta}\\
 \lesssim &\sum_{k \in \N} 2^{dk} 2^{-(n-1)k} 2^{-\theta k}
    \|M\|_{X_\theta}
  \end{split}
  \]
 The sum on the right hand side is absolutely convergent, by the assumption $d < n-1
 + \theta$. This implies that $\int_U \d M$ exists and is independent of the
 choice of $M(\cdot)$ (which makes $\int_U\d M$ well defined).
Moreover the map $M \mapsto \int_U \d M$ is linear and thus continuous.
\end{proof}

\subsection{Weak convergence of the Brouwer degree}
\label{sec:weak-conv-brouw}
Let $U\subset\R^n$ with $\bd\partial U=d\in [n-1,n)$, and $\alpha\in (0,1)$ such
that $n\alpha-d>0$. The following lemma and proposition are taken from
\cite{olbermann2015integrability}; we repeat the proofs for the convenience of
the reader. In the lemma, we use the notation $(A)_\e:=\{x\in\R^n:\dist(x,A)<\e\}$
for $A\subset\R^n$.
\begin{lemma}
\label{lem:bdrydim}
Let $V\subset\R^n$ be open and bounded, $U\subset\subset V$, $n-1<\bd \partial U=d<n$, $0<\alpha<1$ such that $n\alpha>d$,  and $u\in C^{0,\alpha}(V;\R^n)$. Then
\[
\L^n\left((\partial U)_\e\right)\to 0\quad\text{ as }\e\to 0\,.
\]
\end{lemma}

\begin{proof}
We choose  $\delta:=\e^{\alpha^{-1}}$.  Let $x_i\in\partial U$,
$i=1,\dots,k$, be a finite collection of points in the boundary such that
\[
\begin{split}
  \partial U\subset& \bigcup_{i=1}^k B(x_i,\delta)\\
  B(x_i,\delta/5)\cap B(x_j,\delta/5)=&\emptyset \quad\text{ for
  }i,j\in\{1,\dots,k\},i\neq j\,.
\end{split}
\]
The existence of the collection $\{x_i\}$ is assured by the Vitali Covering
Lemma.  Now let  $d<\bar d<n\alpha$. This  implies $\bar H^{\bar d}(\partial
U)=0$ (with $\bar H$ defined in \eqref{eq:20}).
Choosing $\e$ small enough, we may assume that
\[
 k\delta^{\bar d} \leq 1\,, 
\]
We observe that the image of the boundary is covered by  the collection of balls
with centers $u(x_i)$ and radius  $\|u\|_{C^{0,\alpha}}\delta^\alpha$,
\[
\begin{split}
  u(\partial U)\subset& \bigcup_{i=1}^k B\left(u(x_i),\|u\|_{C^{0,\alpha}}\delta^\alpha\right)\\
  = &\bigcup_{i=1}^k B\left(u(x_i),\|u\|_{C^{0,\alpha}}\e\right)\,.
\end{split}
\]
Next we define $c_0=\|u\|_{C^{0,\alpha}}+1$ and obtain
\[
\left(u(\partial U)\right)_\e\subset \bigcup_{i=1}^k B\left(u(x_i),c_0\e\right)\,.
\]
Putting it all together, we have the chain of inequalities
\[
\begin{split}
  \L^n\left(\left(u(\partial U)\right)_\e\right)\leq &k \L^n(B(0,1))(c_0\e)^n\\
  \leq & C(u,n) (\delta^{\alpha n-\bar d})k \tilde\e^{\bar d}\\
  \leq & C(u,n) (\delta^{\alpha n-\bar d}) \\
  \to & 0 \text{ as }\e \to 0\,,
\end{split}
\]
which proves the lemma.
\end{proof}

\begin{proposition}
\label{prop:wcdeg}
  Let $u^j\in C^\infty(U;\R^n)$ with $u^j\to u$ in  $C^{0,\alpha}(U;\R^n)$, and
  $1< p<n\alpha/d$. Then 
\[
\deg(u^j,U,\cdot)\wto \deg(u,U,\cdot)\quad \text{ in }L^p(\R^n)\,.
\]
\end{proposition}
\begin{proof}
Since $u^j$ is smooth, we have the classical change of variables type formula
\[
\int_{U}\varphi(u^j(x))\det Du^j\d x=\int_{\R^n}\varphi(z)\deg(u^j,U,z)\d z\,
\]
for any $\varphi\in L^1_{\mathrm{loc}}(\R^n)$. Let $p'$ be given by
$p^{-1}+(p')^{-1}=1$. We will show
\begin{equation}
\sup_{j\to\infty}\sup\left\{\int_{U}\varphi(u^j(x))\det Du^j\d x:\,\varphi\in
  L^{p'}(\R^n),\,\|\varphi\|_{L^{p'}}\leq 1\right\}<\infty.\label{eq:1}
\end{equation}
This implies that $\deg(u^j,U,\cdot)$ is bounded in $L^p$ and hence there exists
a weakly convergent subsequence. By the  convergence $u\to u^j$ in $C^0$, we
have that $\deg(u^j,U,\cdot)\to \deg(u,U,\cdot)$ pointwise in $S^2\setminus
u(\partial U)$. By Lemma \ref{lem:bdrydim}, $u(\partial U)$ has measure 0, which
implies
\[
\deg(u^j,U,\cdot)\to \deg(u,U,\cdot)\quad  \text{pointwise a.~e.},
\]
and hence we conclude
\[
\deg(u^j,U,\cdot)\wto \deg(u,U,\cdot)\quad  \text{in } L^p(S^n)\,.
\]
 Since we would have obtained the same starting from  any
subsequence of $u^j$,
we get the claim of the
proposition. It remains to show \eqref{eq:1}.\\
\\
Let us fix $\varphi\in L^{p'}(\R^n)$. We define $\zeta \in W^{2,p}(\R^n)$ by 
\[
\Delta \zeta=\varphi\,,
\]
and set $\psi(x)=D\zeta(x)-D\zeta(0)$. By standard elliptic regularity, we have $D\psi\in
L^{p'}(\R^n;\R^{n\times n})$ with 
\[
\|D\psi\|_{L^{p'}}\leq C \|\varphi\|_{L^{p'}}\,.
\]
Since $p<n\alpha/d<n/(n-1)$, we have $p'>n$ and hence, Morrey's inequality implies 
\begin{equation}
[\psi]_{C^{0,1-n/p'}}\leq C \|\varphi\|_{L^{p'}}\,.\label{eq:2}
\end{equation}
Let $\tilde \alpha:=(1-n/p')\alpha$.
We claim that $\psi\circ u^j\in C^{0,\tilde\alpha}$ with 
\begin{equation}
\|\psi\circ u^j\|_{C^{\tilde\alpha}}\leq
  C\|\varphi\|_{L^{p'}}\|u^j\|^{1-n/p'}_{C^{0,\alpha}}\,.\label{eq:6}
  \end{equation}
Indeed we have
\[
\begin{split}
  \sup_x |\psi\circ u^j(x)|=&\sup_x
  |\psi(u^j(x))-\psi(0)|\\
  \leq&[\psi]_{C^{0,1-n/p'}}\left|\sup_x|u^j(x)|-0\right|^{1-n/p'}\\
  \leq& [\psi]_{C^{0,1-n/p'}}\|u^j\|_{C^{0,\alpha}}^{1-n/p'}\,,
  \end{split}
\]
and furthermore
\[
\begin{split}
  |\psi\circ u^j(x)-\psi\circ u^j(x')|\leq &
  [\psi]_{C^{0,1-n/p'}}\left|u^j(x)-u^j(x')\right|^{1-n/p'}\\
  \leq & [\psi]_{C^{0,1-n/p'}} \|u^j\|_{C^{0,\alpha}}^{1-n/p'}|x-x'|^{\alpha
    (1-n/p')}\,.
\end{split}
\]
This proves the claim \eqref{eq:6}.
Next, we recall the identity
\begin{equation}
\varphi\circ u^j \det Du^j =\div \left(\psi\circ u^j\cof Du^j\right)\,,\label{eq:5}
\end{equation}
which can be verified easily by noting $\div\cof Du^j=0$ and $Du^j\cof Du^j=\det
Du^j \id_{n\times n}$. 
For the rest of this proof, we write
\[
M\equiv M_{(0,\dots,0),n-1}\,,
\]
where the right hand side has been defined in Section \ref{sec:weak-pfaff-jacob}.
We recall
\[
M(u_1^j,\dots,u_n^j)= u_1^j\d u_2^j\wedge\dots\wedge\d u_n^j\quad \text{ up to a closed
  $(n-1)$-form}\,,
\]
and express the  identity \eqref{eq:5} using this notation:
\begin{equation}
  \label{eq:4}
\varphi\circ u^j \, \d M (u_1^j,\dots,u_n^j)=\sum_{i=1}^n\d M(u_1^j,\dots,u_{i-1}^j,\psi_i\circ u^j,u_{i+1}^j,\dots,u_n^j)\,.
\end{equation}
Note that
\[
\begin{split}
  (n-1) (\alpha-1)+\tilde \alpha=&\frac{n\alpha}{p'}\\
  =&n\alpha-\frac{n\alpha}{p}\\
  <&n\alpha-d\,.
\end{split}
\]
Hence we may choose $\theta\in (n\alpha/p',n\alpha-d)$,
and we may estimate as follows:
\[
\begin{split}
  \left|\int_U \varphi(u^j(x))\det D u^j(x)\d x\right|=& \left| \int_U
    \sum_{i=1}^n\d M(u_1^j,\dots,u_{i-1}^j,\psi_i\circ
    u^j,u_{i+1}^j,\dots,u_n^j)\right|\\
  \stackrel{\text{Lemma \ref{lem:intwelldef}}}{\lesssim} & \sum_i
    \|M(u_1^j,\dots,u_{i-1}^j,\psi_i\circ
    u^j,u_{i+1}^j,\dots,u_n^j)\|_{X_\theta}\\
    \stackrel{\text{Prop.
        \ref{prop:traceinter}}}{\lesssim}&\sum_i\|u_i^j\|_{C^{0,\tilde\alpha}}
\prod_{k\neq i}\|u^j_k\|_{C^{0,\alpha}}\\
      \stackrel{\eqref{eq:6}}{\lesssim}&\|u^j\|_{C^{0,\alpha}}^{n/p}\|\varphi\|_{L^{p'}}\,.
    \end{split}
\]
This proves \eqref{eq:1} and hence the proposition.
\end{proof}
\begin{remark}
\label{rem:Snrem}
Let $\nu_\e$ be as in Section \ref{sec:appr-c1-alpha}. By using a smooth atlas
on $S^n$, and considering the situation in coordinate charts, we get as an
immediate  consequence of Proposition \ref{prop:wcdeg} that
\[
\deg(\nu_\e,U,\cdot)\wto \deg(\nu,U,\cdot)\quad\text{ in }L^p(S^n)\quad \text{ for } 1<p<n\alpha/d\,.
\]
\end{remark}

\section{Proof of Theorems \ref{thm:cov} and \ref{thm:boundcurv}}
\label{sec:proof-mainthm}
For the proof of Theorem \ref{thm:cov}, we first consider the case
$\varphi=\chi_U$. 
\begin{proposition}
\label{prop:notest}
 Let $U\subset\R^n$ be open and bounded with $\bd \partial U=d\in [n-1,n)$. Let
 $y\in C^{1,\alpha}(U;\R^{n+1})$  be an immersion with $n\alpha>d$, let $\nu \in
 C^{0,\alpha}(U;S^n)$ be the unit normal, and let $\Pf$ be the Pfaffian form
 obtained from the metric $g=Dy^TDy$. Then
\[
\int_U \Pf=\int_{S^n}\deg(\nu,U,z)\d\H^n(z)\,.
\]
\end{proposition}

\begin{proof}
  Let $y_\e,g_\e$ be as in \eqref{eq:16}, and let $\omega,\omega_\e$ be the
  connection one-forms associated to $g,g_\e$ respectively, as in Section
  \ref{sec:appr-c1-alpha}. Furthermore, let $\Omega,\Omega_\e$ be the curvature
  forms associated to $g,g_\e$ respectively.
Since $y_\e$ is smooth, we have
\[
\int_{U}\d\Pi(\omega_\e)=\int_{S^{n}}\deg(\nu_{y_\e},U,z)\d\H^{n}\,.
\]
We are going to pass to the limit $\e\to 0$ on both sides. On the right hand
side, the limit is  $\int_{S^{n}}\deg(\nu,U,z)\d\H^{n}$ by  Remark \ref{rem:Snrem}. 
It remains to show that the limit on the right hand side is 
$\int_{U}\d\Pi(\omega)$. 
By Lemma \ref{lem:gcon} and Lemma \ref{lem:contlem}, we have

\begin{equation}
\left(\omega_i^j(g_\e)\right)_{i,j=1,\dots,n}\to \left(\omega_i^j(g)\right)_{i,j=1,\dots,n} \quad\text{ in } C^{0,\beta}(U;\Lambda^1\R^n)\label{eq:9}
\end{equation}
for all $\beta<2\alpha-1$. Remark \ref{rem:Pirem} implies that for 
\begin{equation}
\theta<n\frac{2\alpha-1}{2}-\left(\frac n2-1\right)=n\alpha-(n-1)\,,\label{eq:18}
\end{equation}
we have
\[
\Pi(\omega(g_\e)))\to \Pi(\omega(g)) \quad\text{ in } X_\theta\text{ as }\e\to 0
\] 
By our assumptions on $d,\alpha$, we may choose $\theta$ such that it fulfills
\eqref{eq:18} and additionally $\theta>d-(n-1)$.
By Lemma \ref{lem:intwelldef}, we get 
\[
\int_U \d\Pi(\omega(g_\e))\to \int_U \d\Pi(\omega(g))\quad\text{ as }\e\to 0\,.
\]
This proves the proposition.
\end{proof}
\begin{remark}
\label{rem:zust}
We note that Proposition \ref{prop:notest} 
could also have been deduced using
the techniques from \cite{MR2745198}. It suffices to note
that by the Gauss-Bonnet-Chern Theorem, the Pfaffian form has the right
structure to apply Theorem 3.2 from \cite{MR2745198}, and hence one can pass to
the limit $\e\to 0$ on the left hand side. 
\end{remark}
\begin{proof}[Proof of Theorem \ref{thm:cov}]
Let $\varphi_k\in C^1(S^n\setminus\nu(\partial U))$ be a sequence that is bounded
uniformly in $L^\infty$ and converges
pointwise to $\varphi$.
It is sufficient to prove the claim for $\varphi_k$, and then apply the
dominated convergence theorem to obtain it for $\varphi$. Hence, from now on, we
may assume $\varphi\in C^1(S^n\setminus \nu(\partial U))$.

We set
\[
A_r:=\{x\in U:\varphi\circ\nu(x)>r\}\,.
\]
Note that $\varphi\circ \nu\in
C^{0,\alpha}(S^n)$, and hence by Lemma \ref{lem:hoelderpreimboxdim} we have

\begin{equation}
\bd\partial A_r=\bd (\varphi\circ\nu)^{-1}(r)\leq n-\alpha \quad \text{ for a.e. }r\in \R\,.\label{eq:12}
\end{equation}
Denoting the characteristic function of $A_r$ by $\chi_{A_r}$, we have for every $x\in U$, 
\[
\varphi\circ\nu(x)=\int_0^\infty 
\chi_{A_r}(x)\d r-\int_{-\infty}^0(1-\chi_{A_r}(x))\d r\,.
\]
By Fubini's Theorem,  we get
\[
\int_{U}\varphi\circ\nu \Pf
=\int_0^\infty\int_{A_r}\Pf \d r-\int_{-\infty}^0\int_{U\setminus A_r}\Pf\d r\,.
\]
Note that 
by \eqref{eq:12} and the assumption $\alpha>n/(n+1)$, we have $n\alpha>\bd\partial A_r$ for almost every $r\in \R$, and hence by
Proposition \ref{prop:notest} , we obtain
\begin{equation}
\begin{split}
  \int_{U}\varphi\circ\nu \Pf=&
  \int_0^\infty\int_{A_r}\Pf \d r-\int_{-\infty}^0\int_{U\setminus A_r}\Pf\d r\\
  =&\int_0^\infty \int_{S^n}\deg(\nu,A_r,z)\d\H^n(z) \d r\\
  &-\int_{-\infty}^0\int_{S^n}\deg(\nu,U \setminus A_r,z)\d\H^n(z)\d r\,.
\end{split}\label{eq:13}
\end{equation}
Now let $\tilde A_r:=\{z\in S^n:\varphi(z)>r\}$. Obviously, $A_r=\nu^{-1}(\tilde
A_r)$, and hence for every $z\in S^n\setminus \nu(\partial A_r)$,

\begin{equation}
\begin{split}
  \deg(\nu,A_r,z)=&\chi_{\tilde A_r}(z)\deg(\nu,U,z)\,,\\
  \deg(\nu,U\setminus A_r,z)=&(1-\chi_{\tilde A_r}(z))\deg(\nu,U,z)\,.
\end{split}\label{eq:14}
\end{equation}
Finally, for every $z\in S^n$, we have
\begin{equation}
\varphi(z)=\int_0^\infty 
\chi_{\tilde A_r}(z)\d r-\int_{-\infty}^0(1-\tilde \chi_{A_r}(z))\d r\,.\label{eq:15}
\end{equation}
Combining \eqref{eq:13}, \eqref{eq:14},  \eqref{eq:15} and   Fubini's Theorem, we obtain
\[
\int_{U}\varphi\circ\nu \Pf=\int_{S^{n}}\varphi(z)\deg(\nu,U,z)\d\H^n(z)\,.
\]
This proves the theorem.
\end{proof}
\begin{proof}[Proof of Theorem \ref{thm:boundcurv}]
The proof works as in \cite{conti2012h}.
We claim that for all $V\subset M$ open with smooth boundary, we have that
\begin{equation} \label{eq:claimproofbdextcurv}
\deg(\nu,V,\cdot) \geq \chi_{\nu(V) \setminus \nu(\partial V)}
\end{equation}

Without loss of generality, we may assume that $V$ is diffeomorphic to an open subset of $\R^{n}$. 
If not, cover $V$ by finitely many open sets $V_{1}, \dots V_{r}$ with smooth boundary that are 
diffeomorphic to an open subset of $\R^{n}$. Set $\tilde V_{i} = V \cap (V_{i}
\setminus \cup_{j < i} V_{j})^{\circ}$, where we have used the notation
$A^\circ$ to denote the interior of a set $A\subset M$. 
Using additivity of the mapping degree we obtain for $z \not \in \nu(\overline{V} \setminus \cup_{i=1}^{r} \tilde V_{i})$,
\[
\begin{split}
\deg(\nu,V,z) =& \deg(\nu,  \cup_{i=1}^{r} \tilde V_{i},z) = \sum_{i=1}^{r} \deg(\nu,\tilde V_{i},z) \\
\geq &\sum_{i=1}^{r} \chi_{\nu(\tilde V_{i}) \setminus \nu(\partial \tilde
  V_{i})}(z)\\
 \geq &\chi_{\nu(V) \setminus \nu(\partial V)}(z)
\end{split}
\] 

But $\H^{n}(\nu(\partial \tilde V_{i})) = 0$ for $i=1,\dots,r$, and since $\deg(\nu,V,\cdot)$ is locally constant 
we obtain the inequality for all $z \in S^{n} \setminus \nu(\partial V)$. \\
\\	
By definition, we have $\deg(\nu,V,z) = 0$ if $z \not \in \nu(V)$. 
For the sake of contradiction, assume that there is $z_{0} \in \nu(V)$ such that $\deg(\nu,V,z_{0}) \leq 0$.
We consider a small disk $D$ around $z_{0}$ with
\[
D \cap \nu(\partial V) = \varnothing
\]
and set $W = \nu^{-1}(D)$. 
Note that $\nu(W) \subset D$ and by continuity of $\nu$, $\nu(\partial W)
\subset \partial D$. Hence,  $\deg(\nu, W, z) = 0$ for $z \in S^{n} \setminus \overline{D}$ and
$\deg(\nu,W,z)=k$ for $z \in D$, where $k$ is some  integer.\\
\\
Let $\varphi\in C^1(S^n)$ with $\varphi\geq 0$,
$\varphi(z_0)>0$ and $\supp \varphi \subset D$. By Theorem \ref{thm:cov} we have
\begin{equation}
\int_{S^{n}}\varphi(z) \deg (\nu,W,z) \d z = \int_{W}\varphi\circ \nu \Pf > 0 \,.\label{eq:19}
\end{equation}
This implies that $k>0$.
By additivity of the degree we have
\[
0 < \deg(\nu,W, z_{0}) = \deg(\nu,V, z_{0}) - \deg(\nu,V\setminus \overline{W}, z_{0}) = \deg(\nu,V, z_{0}) \leq 0 
\]
since, by construction $z_{0} \not \in \nu(V\setminus \overline{W})$. 
But this is a contradiction. \\
\\
Now let $F_{1}, \dots, F_{r} \subset M$ be closed and pairwise disjoint.
We can cover them with disjoint open sets $V_{1}, \dots, V_{r} \subset M$ with smooth boundary 
and use (\ref{eq:claimproofbdextcurv}) and Proposition \ref{prop:notest} to obtain
\[
\begin{split}
\sum_{i=1}^{r} \H^{n}(\nu(F_{i})\setminus \nu(\partial V_{i})) \leq &\sum_{i=1}^{r} \H^{n}(\nu(V_{i})\setminus \nu(\partial V_{i})) \\
\leq& \sum_{i=1}^{r} \int_{S^{n}} \deg(\nu,V_{i},z) \d z \\
= &\sum_{i=1}^{r} \int_{V_{i}} \Pf\\
 \leq &\int_{M} \Pf \,,
\end{split}
\]
which is finite. By our choice of the $V_i$, we have 
$\H^{n}(\nu(\partial V_{i}))=0$ for $i=1,\dots,r$, and hence the theorem is proved.
\end{proof}
\begin{remark}
  As is easily seen from the proof, we could have deduced Theorem \ref{thm:boundcurv} directly from Proposition
  \ref{prop:notest} (without using Theorem \ref{thm:cov}) by choosing
  $\varphi\equiv \chi_D$ in \eqref{eq:19}.
\end{remark}
\bibliographystyle{plain}
\bibliography{rigid}

\end{document}